\documentclass[reqno]{amsart}
\usepackage{amsmath}
\allowdisplaybreaks[3]
\numberwithin{equation}{section}
\numberwithin{equation}{section}

\newtheorem{thm}{\indent Theorem}[section]
\newtheorem{cor}[thm]{\indent Corollary}
\newtheorem{lem}[thm]{\indent Lemma}
\newtheorem{prop}[thm]{\indent Proposition}

\newtheorem{rmk}{{\indent\bf Remark}}[section]
\newtheorem{expl}{{\indent\bf Example}}[section]

\newcommand{\mb}{\mbox}
\newcommand{\hs}{\hspace}

\newcommand{\ol}{\overline}

\newcommand{\ttiny}{\fontsize{5pt}{\baselineskip}\selectfont}
\newcommand{\strl}[2]{\stackrel{\mbox{\ttiny $#1$}}{#2}}

\newcommand{\td}{\tilde}

\newcommand{\fr}{\frac}

\newcommand{\edd}{\end{document}}
\newcommand{\be}{\begin{equation}}
\newcommand{\ee}{\end{equation}}
\newcommand{\bsl}{\backslash}

\newcommand{\lagl}{\langle}
\newcommand{\ragl}{\rangle}
\newcommand{\lmx}{\left(\begin{matrix}}
\newcommand{\rmx}{\end{matrix}\right)}
\newcommand{\ldt}{\left|\begin{matrix}}
\newcommand{\rdt}{\end{matrix}\right|}

\newcommand{\hess}{{\rm Hess\,}}

\newcommand{\tr}{{\rm tr\,}}

\newcommand{\const}{{\rm const}}

\newcommand{\bbr}{{\mathbb R}}

\newcommand{\bbh}{{\mathbb H}}

\newcommand{\bbs}{{\mathbb S}}
\newcommand{\mo}{M\"obius }
\newcommand{\ba}{\begin{array}}
\newcommand{\ea}{\end{array}}
\newcommand{\nnm}{\nonumber}
\newcommand{\beal}{\begin{align}}
\newcommand{\eal}{\end{align}}
\newcommand{\bea}{\begin{eqnarray}}
\newcommand{\eea}{\end{eqnarray}}

\newcommand{\spn}{{\rm Span\,}}

\newcommand{\stx}[2]{\strl{(#1)}{#2}}

\newcommand{\Ba}{\stx{1}{B}{}\!}
\newcommand{\Bb}{\stx{2}{B}{}\!}

\begin{document}

\title[Submanifolds in the unit sphere with
parallel Blaschke tensor]{On the immersed
submanifolds in the unit sphere\\ with parallel Blaschke tensor} 

\author[X. X. Li]{Xingxiao Li} 

\author[H. R. Song]{Hongru Song ${}^*$} 

\dedicatory{}

%
\thanks{ 
Research supported by
Foundation of Natural Sciences of China (No. 11171091, 11371018).}
\address{
Department of Mathematics
\endgraf Henan Normal University \endgraf Xinxiang 453007, Henan
\endgraf P.R. China}
\email{xxl$@$henannu.edu.cn}

\address{
Department of Mathematics
\endgraf Henan Normal University \endgraf Xinxiang 453007, Henan
\endgraf P.R. China} %
\email{yaozheng-shr@163.com}


\thanks{${}^*$The corresponding author.}
\begin{abstract}
As is known, the Blaschke tensor $A$ (a symmetric covariant $2$-tensor) is one of the fundamental \mo invariants in the \mo differential geometry of submanifolds in the unit sphere $\bbs^n$, and the eigenvalues of $A$ are referred  to as the Blaschke eigenvalues. In this paper, we shall prove a classification theorem for immersed
umbilic-free submanifolds in $\bbs^n$ with a parallel Blaschke tensor. For
proving this classification, some new kinds of examples are first defined.
\end{abstract}

\maketitle

\section{Introduction}

Let $\bbs^n(r)$ be the standard $n$-dimensional sphere in the $(n+1)$-dimensional Euclidean space $\bbr^{n+1}$ of radius $r$,
and denote $\bbs^n=\bbs^n(1)$.
Let $\bbh^n(c)$ be the $n$-dimensional hyperbolic space of constant
curvature $c<0$ defined by
$$
\bbh^n(c)=\{y=(y_0,y_1)\in\bbr^{n+1}_1\,;\ \lagl y,y\ragl_1=\fr1c,\
y_0>0\},
$$
where, for any integer $N\geq 2$,
$\bbr^N_1\equiv\bbr_1\times\bbr^{N-1}$ is the $N$-dimensional
Lorentzian space with the standard Lorentzian inner product
$\lagl\cdot,\cdot\ragl_1$ given by
$$
\lagl y,y'\ragl_1=-y_0y'_0+y_1\cdot y'_1,\quad
y=(y_0,y_1),\,y'=(y'_0,y'_1)\in\bbr^N_1
$$
in which the dot ``$\cdot$'' denotes the standard Euclidean inner product
on $\bbr^{N-1}$. From now on, we simply write $\bbh^n$ for $\bbh^n(-1)$.

Denote by $\bbs^n_+$ the hemisphere in $\bbs^n$ whose first coordinate is
positive. Then there are two conformal diffeomorphisms
$$\sigma:\bbr^n\to \bbs^n\bsl\{(-1,0)\}\ \mb{ and }\ \tau:\bbh^n\to \bbs^n_+$$
defined as follows:

\begin{align}
\sigma(u)&=\left(\fr{1-|u|^2}{1+|u|^2},\fr{2u}{1+|u|^2}\right),\quad
u\in \bbr^n, \label{1.1}\\
\tau(y)&=\left(\fr1{y_0},\fr{y_1}{y_0}\right), \quad y=(y_0,y_1)\in
\bbh^n\subset\bbr^{n+1}_1. \label{1.2}
\end{align}

Let $x:M^m\to \bbs^{m+p}$ be an immersed umbilic-free submanifold in
$\bbs^{m+p}$. Then there are four
fundamental \mo invariants of $x$, in terms of the light-cone model established by C. P. Wang in 1998 (\cite{w}), namely, the
\mo metric $g$, the Blaschke tensor $A$, the
\mo second fundamental form $B$ and the \mo form $C$. Since the pioneer work of Wang, there have been obtained many interesting results in the \mo geometry of submanifolds including some important classification theorems of submanifolds with particular \mo invariants, such as, the classification of surfaces
with vanishing \mo forms (\cite{lw2}), that of \mo
isotropic submanifolds (\cite{lwz}), that of hypersurfaces with constant \mo sectional curvature (\cite{gllmw}), that of \mo isoparametric hypersurfaces (\cite{llwz}, \cite{hld}, \cite{ltqw}, etc), and that of hypersurfaces with Blaschke tensors linearly dependent on
the \mo metrics and \mo second  fundamental forms \cite{lw1}, which is later generalized by \cite{lz05} and \cite{clq}, respectively, in two different directions. Here we should remark that, after the classification of all immersed hypersurfaces in $\bbs^{m+1}$ with parallel \mo second fundamental forms (\cite{hl1}), Zhai-Hu-Wang recently proved in \cite{zhw} an interesting theorem which classifies all $2$-codimensional umbilic-free submanifolds in the unit sphere with parallel \mo second  fundamental forms.

To simplify matters, we briefly call an umbilic-free submanifold {\em \mo parallel} if its \mo second fundamental form is parallel.

As for other \mo invariants, it is much natural to study submanifolds in the unit sphere $\bbs^n$ with particular Blaschke tensors. Note that a
submanifold in $\bbs^n$ with vanishing Blaschke tensor also has a
vanishing \mo form, and therefore is a special \mo isotropic
submanifold; any \mo isotropic submanifold is necessarily of parallel Blaschke tensor. Furthermore, all \mo parallel submanifolds also have vanishing \mo forms and parallel Blaschke tensors(\cite{zhw}). So the next natural thing is, of course, to seek a classification of all the submanifolds with parallel Blaschke tensors.

To this direction, the first step is indeed the study of hypersurfaces. In fact, the following theorem has been established:

\begin{thm}[\cite{lz06}]\label{hypcase}
 Let $x:M^m\to \bbs^{m+1}$, $m\geq 2$, be an umbilic-free immersed
hypersurface. If the Blaschke tensor $A$ of $x$ is
parallel, then the \mo form of $x$ vanishes identically and $x$ is either \mo parallel, or \mo isotropic, or \mo equivalent to one of the following examples which have exactly two distinct Blaschke eigenvalues:

$(1)$ one of the minimal hypersurfaces as indicated in
Example $3.2$ of \cite{lz06};

$(2)$ one of the non-minimal hypersurfaces as indicated in
Example $3.3$  of \cite{lz06}.
\end{thm}

As the second step, we shall prove in this paper a classification theorem for all immersed submanifolds in $\bbs^{n}$ with vanishing \mo forms, parallel Blaschke tensors and two distinct Blaschke eigenvalues. To do this, we first need as usual to seek as many as possible examples. As a matter of fact, we successfully construct a new class of immersed submanifolds denoted by ${\rm LS}(m_1,p_1,r,\mu)$ which, as desired,
have vanishing \mo forms and parallel Blaschke tensors with two distinct Blaschke eigenvalues. But they are in general not \mo parallel (see Section \ref{sec3}). It turns out that this class of new examples include those two kinds of examples listed in Theorem \ref{hypcase} that were first introduced in \cite{lz06} (see also \cite{lz07}) and are the only Blaschke isoparametric hypersurfaces (first formally defined in \cite{lz09}) with two distinct Blaschke eigenvalues. Here we should remark that, recently by Li-Wang in \cite{ltw}, any Blaschke isoparametric hypersurfaces with more than two distinct Blaschke eigenvalues must be \mo isoparametric, giving an affirmative solution of the problem originally raised in \cite{lz09} (see also \cite{lp1} and \cite{lp2}). Note that, by a very recent paper \cite{ltqw} and the characterization theorem in \cite{rt}, the \mo isoparametric hypersurfaces which were first introduced by \cite{llwz} have been completely classified. So the above-mentioned theorem by Li-Wang in fact finishes the classification of all the Blaschke isoparametric hypersurfaces and, before this final result, the latest partial classification theorem was proved in \cite{hlz}. By the way, as stated in Theorem \ref{hypcase}, all hypersurfaces with parallel Blaschke tensors necessarily have vanishing \mo forms and thus are special examples of Blaschke isoparametric ones. We also remark that some parallel results for space-like hypersurfaces in the de Sitter space $\bbs^n_1$ have been obtained recently (see \cite{ls1}, \cite{ls2} and the references therein).

Recall that a Riemannian submanifold is said to be {\em pseudo-parallel} if the inner product of its second fundamental form with the mean curvature vector is parallel. In particular, if the second fundamental form is itself parallel, then we simply call this submanifold {\em (Euclidean) parallel}.

Now the main theorem of this paper can be stated as follows:

\begin{thm}\label{main} Let $x:M^m\to \bbs^{m+p}$ be an umbilic-free submanifold immersed in $\bbs^{m+p}$ with parallel Blaschke tensor $A$ and vanishing \mo form $C$. If $x$ has two distinct Blaschke eigenvalues, then it
must be \mo equivalent to one of the following four kinds of immersions:

$(1)$ a non-minimal and umbilic-free pseudo-parallel immersion $\td x:M^m\to \bbs^{m+p}$ with parallel mean curvature and constant scalar curvature, which has two distinct principal curvatures in the direction of the mean curvature vector;

$(2)$ the image under $\sigma$ of a non-minimal and umbilic-free pseudo-parallel immersion $\bar x:M^m\to \bbr^{m+p}$ with parallel mean curvature and constant scalar curvature, which has two distinct principal curvatures in the direction of the mean curvature vector;

$(3)$ the image under $\tau$ of a non-minimal and umbilic-free pseudo-parallel immersion $\bar x:M^m\to \bbh^{m+p}$ with parallel mean curvature and constant scalar curvature, which has two distinct principal curvatures in the direction of the mean curvature vector;

$(4)$ a submanifold ${\rm LS}(m_1,p_1,r,\mu)$ given in Example \ref{expl3.2} for some parameters $m_1,p_1,r,\mu$.
\end{thm}

\begin{rmk}\label{rmk1}\rm In deed, it is directly verified that each of the immersed
submanifolds stated in Theorem \ref{main} has parallel Blaschke tensors and vanishing \mo forms (see Section \ref{sec3}). In fact, some of the examples we shall define in Section \ref{sec3} are new and somewhat more general which can have more than two distinct Blaschke eigenvalues.
\end{rmk}

\begin{rmk}\rm According to \cite{lz09}, an immersed umbilic-free submanifolds in the unit sphere $\bbs^n$ is called {\em Blaschke isoparametric} if (1) the \mo form vanishes identically and (2) all the Blaschke eigenvalues are constant. By carefully checking the argument in this paper, or directly using Proposition A.1 in \cite{ltqw}, one easily finds that we have in fact classified all the Blaschke isoparametric submanifolds in $\bbs^n$ with two distinct Blaschke eigenvalues.
\end{rmk}

\section{Preliminaries}\label{sec2}

Let $x:M^m\to \bbs^{m+p}$ be an immersed umbilic-free submanifold. Denote by $h$ the second fundamental form of $x$
and $H=\fr1m\tr h$ the mean curvature vector field. Define
\be\label{2.1}
\rho=\left(\fr m{m-1}\left(|h|^2-m|H|^2\right)\right)^{\fr12},\quad
Y=\rho(1,x).
\ee
Then $Y:M^m\to \bbr^{m+p+2}_1$ is an immersion
of $M^m$ into the Lorentzian space $\bbr^{m+p+2}_1$ and is called the
canonical lift (or {\em the \mo position vector}) of $x$. The function
$\rho$ given by (2.1) may be called {\em the \mo factor} of the
immersion $x$. We define
$$
C^{m+p+1}_+=\left\{y=(y_0,y_1)\in\bbr_1\times\bbr^{m+p+1}\,;\ \lagl
y,y\ragl_1=0,\ y_0>0\right\}.
$$
Let $O(m+p+1,1)$ be the Lorentzian
group of all elements in $GL(m+p+2;\bbr)$ preserving the standard
Lorentzian inner product $\lagl\cdot,\cdot\ragl_1$ on
$\bbr^{m+p+2}_1$, and $O^+(m+p+1,1)$ be a subgroup of $O(m+p+1,1)$ given
by
\be\label{2.2}
O^+(m+p+1,1)=\left\{T\in O(m+p+1,1)\,;\ T(C^{m+p+1}_+)\subset
C^{m+p+1}_+\right\}.\ee

Then the following theorem is well known.

\begin{thm}\label{wth1} $($\cite{w}$)$ Two submanifolds $x,\td x:M^m\to \bbs^{m+p}$ with
\mo position vectors $Y,\td Y,$ respectively, are \mo equivalent if
and only if there is a $T\in O^+(m+p+1,1)$ such that $\td Y=T(Y)$.
\end{thm}

By Theorem \ref{wth1}, the induced metric
$g=Y^*\lagl\cdot,\cdot\ragl_1=\rho^2dx\cdot dx$ by $Y$ on
$M^m$ from the Lorentzian product $\lagl\cdot,\cdot\ragl_1$ is a \mo
invariant Riemannian metric (cf. \cite{b}, \cite{c}, \cite{w}), and
is called the \mo metric of $x$. Using the vector-valued function
$Y$ and the Laplacian $\Delta$ of the metric $g$, one can define
another important vector-valued function $N:M^m\to\bbr^{m+p+2}_1$, called {\em the \mo biposition vector}, by
\be\label{2.3}
N=-\fr1m\Delta Y-\fr1{2m^2}\lagl\Delta Y,\Delta Y\ragl_1Y.
\ee
Then it is verified that the \mo position vector $Y$ and the \mo
biposition vector $N$ satisfy the following identities \cite{w}:
\begin{align}
&\lagl \Delta Y,Y\ragl_1=-m,\quad \lagl \Delta Y,dY\ragl_1=0,\quad
\lagl\Delta Y,\Delta Y\ragl_1=1+m^2\kappa, \label{2.4}\\
&\lagl Y,Y\ragl_1=\lagl N,N\ragl_1=0,\quad \lagl Y,N\ragl_1=1,
\label{2.5}
\end{align}
where $\kappa$ denotes the normalized scalar curvature of the \mo
metric $g$.

Let $V\to M^m$ be the vector subbundle of the trivial Lorentzian
bundle $M^m\times\bbr^{m+p+2}_1$ defined to be the orthogonal
complement of $\bbr Y\oplus \bbr N\oplus Y_*(TM^m)$ with respect to
the Lorentzian product $\lagl\cdot,\cdot\ragl_1$. Then $V$ is called
the \mo normal bundle of the immersion $x$. Clearly, we have the
following vector bundle decomposition:
\be\label{2.6}
M^m\times\bbr^{m+p+2}_1=\bbr Y\oplus \bbr N\oplus Y_*(TM^m)\oplus
V.
\ee

Now, let $T^\bot M^m$ be the normal bundle of the immersion
$x:M^m\to \bbs^{m+p}$. Then the mean curvature vector field $H$ of $x$
defines a bundle isomorphism $\Phi:T^\bot M^m\to V$ by
\be\label{2.7}
\Phi(e)=\left(H\cdot e,(H\cdot e)x+e\right)\quad\text{for any } e\in
T^\bot M^m.
\ee
It is known that $\Phi$ preserves the inner products as well as
the connections on $T^\bot M^m$ and $V$ (\cite{w}).

To simplify notations, we make the following conventions on the
ranges of indices used frequently in this paper:
\be\label{2.8}
1\leq i,j,k,\cdots\leq m,\quad m+1\leq
\alpha,\beta,\gamma,\cdots\leq m+p.
\ee

For a local orthonormal  frame field $\{ e_i\}$ for the induced
metric $dx\cdot dx$ with dual $\{\theta^i\}$ and for an orthonormal
normal frame field $\{ e_\alpha\}$ of $x$, we set
\be\label{2.9}
E_i=\rho^{-1}e_i,\quad \omega^i=\rho\theta^i,\quad
E_\alpha=\Phi(e_\alpha).
\ee
Then $\{E_i\}$ is a local orthonormal frame field on $M^m$ with respect to the \mo metric
$g$, $\{\omega^i\}$ is the dual of $\{E_i\}$, and $\{E_\alpha\}$ is
a local orthonormal frame field of the \mo normal bundle $V\to M$. Clearly,
$\{Y,N,Y_i:=Y_*(E_i),E_\alpha\}$ is a moving frame of $\bbr^{m+p+2}_1$ along $M^m$. If the basic \mo invariants $A$, $B$ and $C$ are respectively written as
\be\label{2.10}
A=\sum
A_{ij}\omega^i\omega^j,\quad  B=\sum B^\alpha_{ij}\omega^i\omega^j
E_\alpha,\quad C=\sum  C^\alpha_i\omega^i E_\alpha,
\ee
then we have the following equations of motion (\cite{w}):
\begin{align}
dY=&\sum Y_i\omega^i,\quad dN=\sum A_{ij}\omega^jY_i+C^\alpha_i\omega^iE_\alpha,\label{2-6}\\
dY_i=&-\sum A_{ij}\omega^jY-\omega^iN+\sum \omega^j_iY_j+\sum B^\alpha_{ij}\omega^jE_\alpha,\label{2-7}\\
dE_\alpha=&-\sum C^\alpha_i\omega^iY-\sum B^\alpha_{ij}\omega^jY_i+\sum \omega^\beta_\alpha E_\beta,\label{2-8}
\end{align}
where $\omega^j_i$ are the Levi-Civita connection forms of the \mo metric $g$ and $\omega^\beta_\alpha$ are the (\mo) normal connection forms of $x$. Furthermore,
by a direct computation one can find the following local expressions (\cite{w}):
\begin{align}
A_{ij}=&- \rho^{-2}\left(\hess_{ij}(\log \rho)- e_i(\log \rho)
e_j(\log \rho)
-\sum  H^\alpha h^\alpha_{ij}\right) \nnm\\
&-\fr12\rho^{-2}\left(|d\log \rho|^2-1+| H|^2\right)\delta_{ij},\label{2.12}\\
B^\alpha_{ij}=&\ \rho^{-1}\left( h^\alpha_{ij}-
H^\alpha\delta_{ij}\right), \label{2.13}\\
C^\alpha_i=&-\rho^{-2}\left(H^\alpha_{,i}+\sum(
h^\alpha_{ij}-H^\alpha\delta_{ij})
e_j(\log\rho)\right), \label{2.11}
\end{align}
in which the subscript ``$,i$'' denotes the covariant derivative
with respect to the induced metric $d x\cdot d x$ and in the
direction $e_i$.

\begin{rmk}\label{rmk2}\rm For an umbilic-free immersion $\bar x:M^m\to\bbr^{m+p}$ (resp. $\bar x:M^m\to\bbh^{m+p}$), a \mo factor $\bar\rho$, a \mo invariant metric $\bar g$ and other \mo invariants
$\bar A,\bar B, \bar C$ are defined similarly. As indicated in \cite{w} and \cite{lwz}, while the
corresponding components $\bar B_{ij}$ (resp. $\bar C^\alpha_i$) of $\bar B$ (resp. $\bar C$) have the same expressions as \eqref{2.13} (resp. \eqref{2.11}), the components $\bar A_{ij}$ of $\bar A$ has a
slightly different expression from \eqref{2.12}:
\begin{align}
\bar A_{ij}=&- \bar\rho^{-2}\left(\hess_{ij}(\log \bar\rho)-
e_i(\log \bar\rho) e_j(\log \bar\rho)
-\sum \bar H^\alpha \bar h^\alpha_{ij}\right)\nnm\\
&-\fr12\bar\rho^{-2}\left(|d\log \bar\rho|^2+|\bar
H|^2\right)\delta_{ij}\label{2.12'}\\
\mb{(resp.}\hs{2cm} &\nnm\\
\bar A_{ij}=&- \bar\rho^{-2}\left(\hess_{ij}(\log \bar\rho)-
e_i(\log \bar\rho) e_j(\log \bar\rho)
-\sum \bar H^\alpha \bar h^\alpha_{ij}\right)\nnm\\
&-\fr12\bar\rho^{-2}\left(|d\log \bar\rho|^2+1+|\bar
H|^2\right)\delta_{ij}\mb{ )}\label{2.12''}
\end{align}
\end{rmk}

Denote, respectively, by $R_{ijkl}$, $R^\bot_{\alpha\beta ij}$ the components
of the \mo Riemannian curvature tensor and the
curvature operator of the \mo normal bundle
with respect to the tangent frame field $\{E_i\}$ and the \mo normal frame field $\{E_\alpha\}$. Then we have (\cite{w})
\begin{align} \tr
A&=\fr1{2m}(1+m^2\kappa),\quad\tr B=\sum
B^\alpha_{ii}E_\alpha=0,\quad |B|^2=\sum
(B^\alpha_{ij})^2=\fr{m-1}{m}. \label{2.14}
\\
&R_{ijkl}=\sum(B^\alpha_{il}B^\alpha_{jk}
-B^\alpha_{ik}B^\alpha_{jl})+A_{il}\delta_{jk}-A_{ik}\delta_{jl}
+A_{jk}\delta_{il}-A_{jl}\delta_{ik}. \label{2.15}\\
&\hs{2.6cm}R^\bot_{\alpha\beta ij}=\sum(B^\alpha_{jk}B^\beta_{ik}-B^\alpha_{ik}B^\beta_{jk}
).\label{2-16}
\end{align}

We should remark that both equations \eqref{2.15} and \eqref{2-16} have the opposite sign from those in \cite{w} due to the different notations
of the Riemannian curvature tensor. Furthermore, let
$A_{ijk}$, $B^\alpha_{ijk}$ and $C^\alpha_{ij}$ denote, respectively,
the components with respect to the frame fields $\{E_i\}$ and
$\{E_\alpha\}$ of the covariant derivatives of $A$, $B$ and $C$, then the following Ricci identities hold (\cite{w}):
\begin{align}
A_{ijk}-A_{ikj}=&\sum
(B^\alpha_{ik}C^\alpha_j-B^\alpha_{ij}C^\alpha_k), \label{2.17}\\
B^\alpha_{ijk}-B^\alpha_{ikj}=&\delta_{ij}C^\alpha_k
-\delta_{ik}C^\alpha_j, \label{2.18}\\
C^\alpha_{ij}-C^\alpha_{ji}
=&\sum(B^\alpha_{ik}A_{kj}-B^\alpha_{kj}A_{ki}).  \label{2.16}
\end{align}

Denote by $R_{ij}$ the components of the Ricci curvature. Then
by taking trace in \eqref{2.15} and \eqref{2.18}, one obtains
\begin{align}
&R_{ij}=-\sum B^\alpha_{ik}B^\alpha_{kj}+\delta_{ij}\tr
A+(m-2)A_{ij}, \label{2.19}\\
&(m-1)C^\alpha_i=-\sum B^\alpha_{ijj}. \label{2.20}
\end{align}

Moreover, for the higher order covariant derivatives $B^\alpha_{ij\cdots kl}$, we have the following Ricci identities:
\be\label{2-21}
B^\alpha_{ij\cdots kl}-B^\alpha_{ij\cdots lk}=\sum B^\alpha_{qj\cdots}R_{iqkl} +\sum B^\alpha_{iq\cdots}R_{jqkl}+\cdots-\sum B^\beta_{ij\cdots}R^\bot_{\beta\alpha kl}.
\ee

By \eqref{2.14}, \eqref{2.19} and \eqref{2.20}, if $m\geq
3$, then the Blaschke tensor $A$ and the \mo form $C$ are
determined by the \mo metric $g$, \mo second fundamental form
$B$ and the (\mo) normal connection of $x$. Thus the following theorem holds:

\begin{thm}[cf. \cite{w}]\label{wth} Two submanifolds $x:M^m\to
\bbs^{m+p}$ and $\td x:\td M^m\to \bbs^{m+p}$, $m\geq 3$, are \mo
equivalent if and only if they have the same \mo metrics, the same \mo second fundamental forms and the same (\mo) normal connections.
\end{thm}

\section{The new examples}\label{sec3}

Before proving the main theorem, we need to find as many as possible examples of submanifolds in the unit sphere $\bbs^{m+p}$ with parallel Blaschke tensors and with two distinct Blaschke eigenvalues. First we note that, by Zhai-Hu-Wang (\cite{zhw}), all \mo parallel submanifolds in $\bbs^{m+1}$ necessarily have parallel Blaschke tensors. Examples of this kind of submanifolds are listed in \cite{zhw}. In this section we define a new class of examples with parallel Blaschke tensors which are in general not \mo parallel.

{\expl\label{expl3.1}\rm Here we are to examine the following three classes of submanifolds that meet the conditions of Theorem \ref{main}.

(1) Let $\td x:M^m\to\bbs^{m+p}$ be an umbilic-free pseudo-parallel submanifolds with parallel mean curvature $\td H$ and constant scalar curvature $\td S$.

Since the mean curvature $\td H$ is parallel (implying that $|\td H|^2=\const$) and the scalar curvature $\td S$ is constant, by the Gauss equation and \eqref{2.1}, we find that the \mo factor $\td\rho$ is also a constant. It follows by \eqref{2.11} that the \mo form $\td C\equiv 0$. Note that, by $\td\rho=\const$, the parallel of tensors with respect to the induced metric $d\td x^2$ and the \mo metric $\td g$ are exactly the same. Consequently, by \eqref{2.12}, the Blaschke tensor $\td A$ of $\td x$ is parallel since $\td x$ is pseudo-parallel.

Clearly, $\td x$ has two distinct Blaschke eigenvalues if and only if it is not minimal and has two distinct principal curvatures in the direction of the mean curvature vector $\td H$. Note that $\td x$ is \mo isotropic, or equivalently, $\td x$ has only one distinct Blaschke eigenvalue, if and only it is minimal (\cite{lwz}).

(2) Let $\bar x:M^m\to\bbr^{m+p}$ be an umbilic-free pseudo-parallel submanifolds with parallel mean curvature $\bar H$ and constant scalar curvature $\bar S$.

As in (1), since the mean curvature $\bar H$ is parallel (in particular $|\bar H|^2=\const$), and the scalar curvature $\bar S$ is constant, we know from \eqref{2.1} and the Gauss equation of $\bar x$ that the \mo factor $\bar\rho$ is once again a constant. Thus, by \eqref{2.11}, the \mo form $\bar C\equiv 0$. Consequently, by \eqref{2.12'}, the Blaschke tensor $\bar A$ of $\bar x$ is parallel since $\bar x$ is pseudo-parallel. Furthermore, $\bar x$ is of two distinct Blaschke eigenvalues if and only if it is not minimal with two distinct principal curvatures in the direction of the mean curvature vector $\bar H$.

Define $\td x:=\sigma\circ\bar x$. Then by \cite{lwz}, $\td x$ has a parallel Blaschke tensor. Furthermore, it is of two distinct Blaschke eigenvalues if and only if $\bar x$ is not minimal with two distinct principal curvatures in the direction of the mean curvature vector.

(3) Let $\bar x:M^m\to\bbh^{m+p}$ be an umbilic-free pseudo-parallel submanifold with parallel mean curvature $\bar H$ and constant scalar curvature $\bar S$. Then, as in (2), we obtain that $\td x:=\tau\circ\bar x$ has a parallel Blaschke tensor; it has two distinct Blaschke eigenvalues if and only if $\bar x$ is not minimal and has two distinct principal curvatures in the direction of the mean curvature vector.

\begin{rmk}\label{rmk3.1}\rm It is not hard to see that the submanifold $\td x$ in (1) is \mo parallel if and only if it is (Euclidean) parallel; the submanifold $\td x$ in (2) is \mo parallel if and only if the corresponding submanifold $\bar x:M^m\to\bbr^{m+p}$ is (Euclidean) parallel; and the submanifold $\td x$ in (3) is \mo parallel if and only if the corresponding submanifold $\bar x:M^m\to\bbh^{m+p}$ is (Euclidean) parallel.\end{rmk}

{\expl\label{expl3.2}\rm Submanifolds ${\rm LS}(m_1,p_1,r,\mu)$ with parameters $m_1,p_1,r,\mu$.

Fix the dimension $m\geq 3$ and the codimension $p\geq 1$. We start with a multiple parameter data $(m_1,p_1,r,\mu)$ where $m_1,p_1$ are integers satisfying
$$1\leq m_1\leq m-1,\quad  0\leq p_1\leq p,$$
and $r>0$, $\mu\in [0,1]$ are real numbers. Denote $m_2:=m-m_1$ and $p_2=p-p_1$.}

Let $\td y=(\td y_0,\td y_1):M_1\to
\bbh^{m_1+p_1}\left(-\fr1{r^2}\right)\subset\bbr^{m_1+p_1+1}_1$ be an immersed
minimal submanifold of dimensional $m_1$ with constant scalar
curvature
\be\label{3.20}\td S_1=-\fr{m_1(m_1-1)}{r^2}-\fr{m-1}{m}\mu,\ee
and
\be\label{3.21}
\td y_2: M_2\to \bbs^{m_2+p_2}(r)\subset\bbr^{m_2+p_2+1}
\ee
be an immersed minimal submanifold of dimension $m_2$
with constant scalar curvature
\be\label{3.1}\td S_2=\fr{m_2(m_2-1)}{r^2}-\fr{m-1}{m}(1-\mu).\ee
Clearly,
\be\label{S1+S2}
\td S_1+\td S_2=\fr{-m_1(m_1-1)+m_2(m_2-1)}{r^2}-\fr{m-1}{m}.
\ee

Set
\be\label{3.22}
\td M^m=M_1\times M_2,\quad\td Y=(\td y_0,\td y_1,\td
y_2).
\ee
Then $\td Y:\td M^m\to\bbr^{m+p+2}_1$ is an immersion satisfying $\lagl
\td Y,\td Y\ragl_1=0$ and has the induced Riemannian metric
$$g=\lagl d\td Y,d\td Y\ragl_1=-d\td y_0^2+d\td y^2_1+d\td y^2_2.$$
Obviously, as a Riemannian
manifold, we have
\be\label{3.23}
(\td M^m,g)=(M_1,\lagl d\td y,d\td y\ragl_1)\times
\left(M_2,d\td y^2_2\right).\ee
Define
\be\label{3.24}
\td x_1=\fr{\td y_1}{\td y_0},\quad \td
x_2=\fr{\td y_2}{\td y_0}, \quad \td x=(\td x_1,\td x_2).
\ee
Then $\td x^2=1$ and $\td x:\td M^m\to \bbs^{m+p}$ is an immersed submanifold which we denote by ${\rm LS}(m_1,p_1,r,\mu)$ for simplicity. Since
\be\label{3.25}
d\td x=-\fr{d\td y_0}{\td y^2_0}(\td y_1,\td y_2) +\fr1{\td
y_0}(d\td y_1,d\td y_2), \ee
the induced metric $\td g=d\td x\cdot d\td x$ on $\td M^m$ is related to $g$ by
\be\label{3.26}
\td g=\td y^{-2}_0(-d\td y^2_0+d\td y^2_1+d\td y^2_2)=\td
y^{-2}_0g.
\ee

Let $$\{\bar e_\alpha; m+1\leq\alpha\leq m+p_1\}\ \mb{ (resp. }\ \{\bar e_\alpha; m+p_1+1\leq\alpha\leq m+p\})$$ be an orthonormal normal frame field of $\td y$ (resp. $\td y_2$) with
$$\bar e_\alpha=(\bar e_{\alpha 0},\bar e_{\alpha 1})\in\bbr^1_1\times\bbr^{m_1+p_1}\equiv\bbr^{m_1+p_1+1}_1, \mb{ for }\alpha=m+1,\cdots,m+p_1.$$
Define
\begin{align}
\td e_\alpha=&(\bar e_{\alpha 1},0)-\bar e_{\alpha 0} \td x\in\bbr^{m_1+p_1}\times\bbr^{m_2+p_2+1}\equiv\bbr^{m+p+1},\ \mb{for }\alpha=m+1,\cdots,m+p_1;\label{3-10}\\
\td e_\alpha=&(0,\bar e_\alpha)\in\bbr^{m_1+p_1}\times\bbr^{m_2+p_2+1}\equiv\bbr^{m+p+1},\ \mb{for }\alpha=m+p_1+1,\cdots,m+p.\label{3-11}
\end{align}
Then $\{\td e_\alpha;\ m+1\leq\alpha\leq m+p\}$ is an orthonormal normal frame field of ${\rm LS}(m_1,p_1,r,\mu)$.

Hence, by \eqref{3.25}, for $\alpha=m+1,\cdots,m+p_1$
\begin{align} d \td e_\alpha\cdot d\td x=&\ (d\bar e_{\alpha1},0)\cdot d\td x-d\bar e_{\alpha 0}\td x d\td x-\bar e_{\alpha 0}d\td x^2\nnm\\
=&\ \td y^{-1}_0(-d\bar e_{\alpha 0}d\td y_0+d\bar e_{\alpha 1}\cdot d\td y_1)-\bar e_{\alpha 0}\td
y^{-2}_0g;\label{3.27}
\end{align}
while for $\alpha=m+p_1+1,\cdots,m+p$,
$$
d \td e_\alpha\cdot d\td x=\td y^{-1}_0(d\bar e_\alpha\cdot d\td y_2).
$$
Consequently, if we denote by
$$
\bar h_{M_1}=\sum_{\alpha=m+1}^{m+p_1}\bar h^\alpha \bar e_\alpha,\quad \bar h_{M_2}=\sum_{\alpha=m+p_1+1}^{m+p}\bar h^\alpha \bar e_\alpha
$$
the second fundamental forms of $\td y$ and $\td y_2$, respectively, then the second fundamental form
$$\td h=\sum_{\alpha=m+1}^{m+p}\td h^\alpha \td e_\alpha$$ of ${\rm LS}(m_1,p_1,r,\mu)$ is given in terms of
$\bar h_{M_1}$, $\bar h_{M_2}$ and the metric $g$ as follows:
\begin{align}
\td h^\alpha=&-d \td e_\alpha\cdot d\td x=y^{-1}_0\bar h^\alpha+\bar e_{\alpha 0}\td y^{-2}_0g, \quad \mb{for }\alpha=m+1,\cdots,m+p_1;\label{3.29}\\
\td h^\alpha=&-d \td e_\alpha\cdot d\td x=\td
y^{-1}_0\bar h^\alpha,\quad \mb{for }\alpha=m+p_1+1,\cdots,m+p.\label{3.29-1}
\end{align}

Let $$\{E_i\,;1\leq i\leq m_1\}\ \mb{ (resp. }\ \{E_i\,;m_1+1\leq i\leq
m\})$$
be a local orthonormal frame field for $(M_1,\lagl d\td y,d\td y\ragl_1)$ (resp.
for $(M_2,d\td y^2_2)$). Then $\{E_i\,;1\leq i\leq m\}$ is a local
orthonormal frame field for $(M^m,g)$. Put $\td e_i=\td y_0E_i$,
$i=1,\cdots,m$. Then $\{\td e_i\,;1\leq i\leq m\}$ is a local
orthonormal frame field for $(M^m,\td g)$. Thus for $\alpha=m+1,\cdots,m+p_1$,
\be\label{3.30}
\left\{\aligned
\td h^\alpha_{ij}=&\ \td h^\alpha(\td e_i,\td e_j)=\td y^2_0\td h^\alpha(E_i,E_j)=\td y_0
\bar h^\alpha(E_i,E_j)+\bar e_{\alpha 0}\,g(E_i,E_j)\\
=&\ \td y_0\bar h^\alpha_{ij}+\bar e_{\alpha 0}\delta_{ij},
\quad \mb{when\ } 1\leq i,j\leq m_1,\\
\td h^\alpha_{ij}=&\ \bar e_{\alpha 0}\delta_{ij},\quad \mb{when\ }i>m_1\ \mb{or\
}j>m_1;
\endaligned\right.
\ee
and for $\alpha=m+p_1+1,\cdots,m+p$,
\be\label{3.9}
\left\{\aligned
\td h^\alpha_{ij}=&\ \td h^\alpha(\td e_i,\td e_j)=\td y^2_0\td h^\alpha(E_i,E_j)=\td y_0
\bar h^\alpha(E_i,E_j)=\td y_0\bar h^\alpha_{ij},\\
&\mb{when\ } m_1+1\leq i,j\leq m,\\
\td h^\alpha_{ij}=&\ 0,\quad \mb{when\ }i\leq m_1\ \mb{or\ }j\leq m_1.
\endaligned\right.
\ee
By using the minimality of both $\td y$ and $\td y_1$, the mean
curvature
$$\td H=\fr1m\sum_{\alpha=m+1}^{m+p}\sum_{i=1}^m\td h^\alpha_{ii}\td e_\alpha$$
of ${\rm LS}(m_1,p_1,r,\mu)$ is given by
\begin{align}
\td H^\alpha=&\fr1m\sum_{i=1}^m\td h^\alpha_{ii}
=\fr{\td y_0}m\sum_{i=1}^{m_1}\bar h^\alpha_{ii}+\bar e_{\alpha 0}
=\bar e_{\alpha 0},\quad\mb{ for }m+1\leq\alpha\leq m+p_1;\label{3.31}\\
\td H^\alpha=&\fr1m\sum_{i=1}^m\td h^\alpha_{ii}
=\fr{\td y_0}m\sum_{i=m_1+1}^{m}\bar h^\alpha_{ii}=0,\quad\mb{ for }m+p_1+1\leq\alpha\leq m+p.\label{3.31-1}
\end{align}
From \eqref{S1+S2}, \eqref{3.30}--\eqref{3.31-1} and the Gauss equations of $\td y$ and $\td y_2$, we find
\begin{align}
|\td h|^2=&\td y^2_0\sum_{\alpha=m+1}^{m+p_1}\sum_{i,j=1}^{m_1}(\bar h^\alpha_{ij})^2+m\sum_{\alpha=m+1}^{m+p_1}(\bar e_{\alpha 0})^2+\td y^2_0\sum_{\alpha=m+p_1+1}^{m+p}\sum_{i,j=m_1+1}^{m}(\bar h^\alpha_{ij})^2\nnm\\
=&\fr{m-1}{m}\td y^2_0+m\sum_{\alpha=m+1}^{m+p_1}(\bar e_{\alpha 0})^2,\\
|\td H|^2=&\sum_{\alpha=m+1}^{m+p_1}(\td H^\alpha)^2 +\sum_{\alpha=m+p_1+1}^{m+p}(\td H^\alpha)^2=\sum_{\alpha=m+1}^{m+p_1}(\bar e_{\alpha 0})^2.
\end{align}
It then follows that
$$|\td h|^2-m|\td
H|^2=\fr{m-1}{m}\td y^2_0>0,$$
implying that $\td x$ is umbilic-free, and the \mo factor $\td \rho=\td y_0$. So $\td Y$ is the \mo position of ${\rm LS}(m_1,p_1,r,\mu)$. Consequently,
the \mo metric of ${\rm LS}(m_1,p_1,r,\mu)$ is nothing but $\lagl d\td Y,d\td Y
\ragl_1=g$. Furthermore, if we denote by $\{\omega^i\}$ the local coframe field on $M^m$ dual to $\{E_i\}$, then the \mo second fundamental form
$$
\td B=\sum_{\alpha=m+1}^{m+p}\td B^\alpha \Phi(\td e_\alpha)\equiv
\sum_{\alpha=m+1}^{m+p}\td B^\alpha_{ij}\omega^i\omega^j \Phi(\td e_\alpha)
$$ of ${\rm LS}(m_1,p_1,r,\mu)$ is given by
\begin{align}
\td B^\alpha=&\td \rho^{-1}\sum(\td h^\alpha_{ij}-\td H^\alpha\delta_{ij})\omega^i\omega^j
=\sum_{i,j=1}^{m_1}\bar h^\alpha_{ij}\omega^i\omega^j,\nnm\\
&\ \mb{ for }\alpha=m+1,\cdots,m+p_1; \label{3.32}\\
\td B^\alpha=&\td \rho^{-1}\sum(\td h^\alpha_{ij}-\td H^\alpha\delta_{ij})\omega^i\omega^j
=\sum_{i,j=m_1+1}^{m}\bar h^\alpha_{ij}\omega^i\omega^j,\nnm\\
&\ \mb{ for }\alpha=m+p_1+1,\cdots,m+p, \label{3.32_1}
\end{align}
or, equivalently
\be\label{B}
\td B^\alpha_{ij}=\begin{cases} \bar h^\alpha_{ij},& \mb{if } m+1\leq\alpha\leq m+p_1,\ 1\leq i,j\leq m_1\\
&\mb{or } m+p_1+1\leq\alpha\leq m+p,\ m_1+1\leq i,j\leq m,\\
0,&\mb{otherwise.}
\end{cases}
\ee

On the other hand, since the \mo metric $g$ is the direct product of $\lagl d\td y,d\td y\ragl_1$ and $d\td y_2\cdot d\td y_2$, one finds by the minimality and the Gauss equations of $\td y$
and $\td y_2$ that the Ricci tensor of $g$ is given as
follows:
\begin{align}
R_{ij}=&-\fr{m_1-1}{r^2}\delta_{ij}-\sum_\alpha\sum_{k=1}^{m_1}\bar h^\alpha_{ik}\bar h^\alpha_{kj},\quad\mb{if\
} 1\leq
i,j\leq m_1, \label{3.33}\\
R_{ij}=&\ \fr{m_2-1}{r^2}\delta_{ij}-\sum_\alpha\sum_{k=m_1+1}^{m}\bar h^\alpha_{ik}\bar h^\alpha_{kj},\quad\mb{if\ } m_1+1\leq
i,j\leq m, \label{3.34}\\
R_{ij}=&\ 0,\quad\mb{otherwise}. \label{3.35}
\end{align}
But by the definitions of $\td y$ and $\td y_2$, the normalized scalar curvature $\kappa$ of $g$ is given by (see \eqref{S1+S2})
\begin{align*}\kappa=&\fr{-m_1(m_1-1)+m_2(m_2-1)}{m(m-1)r^2}-\fr{1}{m^2}.
\end{align*}
Thus
\be\label{3.37}
\tr A=\fr1{2m}(1+m^2\kappa)=\fr{-m_1(m_1-1)+m_2(m_2-1)}{2(m-1)r^2}.
\ee
Since $m\geq 3$, it follows by \eqref{2.19}, \eqref{B}--\eqref{3.35} that the Blaschke tensor of ${\rm LS}(m_1,p_1,r,\mu)$ is given by $A=\sum
A_{ij}\omega^i\omega^j$ where, for $1\leq i,j\leq m_1$,
\begin{align}
A_{ij}=&\fr1{m-2}\left(R_{ij}+\sum_{\alpha,k}\td B^\alpha_{ik}\td B^\alpha_{kj}-\delta_{ij}\tr A\right)\nnm\\
=&-\fr1{2r^2}\delta_{ij};\label{3.38}
\end{align}
and similarly,
\begin{align}
A_{ij}=&\ \fr{1}{2r^2}\delta_{ij},\quad\mb{for } m_1+1\leq
i,j\leq m, \label{3.39}\\
A_{ij}=&\ 0,\quad\mb{otherwise}. \label{3.40}
\end{align}

Therefore, $A$ has two distinct eigenvalues
\be\label{3.41}
\lambda_1=-\lambda_2=-\fr{1}{2r^2}.
\ee
Thus ${\rm LS}(m_1,p_1,r,\mu)$ is of parallel Blaschke tensor $A$ since $\omega^j_i=0$ for $A_{ii}\neq A_{jj}$.

\begin{prop} Submanifolds ${\rm LS}(m_1,p_1,r,\mu)$ defined in Example \ref{expl3.2} are of vanishing \mo form; they are \mo parallel if and only if both
$$\td y:M_1\to \bbh^{m_1+p_1}\left(-\fr1{r^2}\right)\ \mb{ and }\ \td y_2:M_2\to \bbs^{m_2+p_2}(r)$$
are parallel as Riemannian submanifolds. Furthermore, if it is the case, then $x(M_1)$ is isometric to the totally geodesic hyperbolic space $\bbh^{m_1}\left(-\fr1{r^2}\right)$ and $\td y$ can be taken as the standard embedding of $\bbh^{m_1}\left(-\fr1{r^2}\right)$ in $\bbh^{m_1+p_1}\left(-\fr1{r^2}\right)$.
\end{prop}

\begin{proof} Firstly, if we denote by $$\bar\omega^\beta_\alpha,\quad m+1\leq\alpha,\beta\leq m+p_1\ \mb{ (resp. }\ m+p_1+1\leq\alpha,\beta\leq m+p),$$ the normal connection forms of $\td y$ (resp. $\td y_2$) with respect to the normal frame field
$$\{\bar e_\alpha,\ m+1\leq\alpha\leq m+p_1\}\ \mb{ (resp. } \ \{\bar e_\alpha,\ m+p_1+1\leq\alpha\leq m+p\}),$$
then by definitions \eqref{3-10} and \eqref{3-11} of the normal frame field $\{\td e_\alpha,\ m+1\leq\alpha\leq m+p\}$ for ${\rm LS}(m_1,p_1,r,\mu)$, we easily find that the normal connection forms $\td \omega^\beta_\alpha$, $m+1\leq\alpha,\beta\leq m+p$, with respect to $\{\td e_\alpha,\ m+1\leq\alpha\leq m+p\}$ are as follows:
\be\label{3-37}
\td\omega^\beta_\alpha=\begin{cases} \bar\omega^\beta_\alpha,&\mb{for both } m+1\leq \alpha,\beta\leq m+p_1, \\
&\mb{and } m+p_1+1\leq \alpha,\beta\leq m+p;\\
0, &\mb{otherwise.}
\end{cases}
\ee

On the other hand, note that the \mo metric of ${\rm LS}(m_1,p_1,r,\mu)$ is the direct product of the induced metrics of $\td y$ and $\td y_2$, and the bundle map $\Phi:T^\bot M\to V$ (see Section \ref{sec2}) keeps invariant both the metric and the connections. Therefore the first conclusion comes directly from \eqref{2.20} together with the fact that both
$$\sum_{\alpha=m+1}^{m+p_1} \sum_{i,j=1}^{m_1}B^\alpha_{ij}\omega^i\omega^jE_\alpha \ \mb{ and }\
\sum_{\alpha=m+p_1+1}^{m+p} \sum_{i,j=m_1+1}^{m}B^\alpha_{ij}\omega^i\omega^jE_\alpha$$
are normal-vector-valued Codazzi tensors on $M_1$ and $M_2$, respectively; The second conclusion is
easily seen true by \eqref{B} and \eqref{3-37}; And the third conclusion of the proposition comes from the fact that any connected minimal submanifolds with parallel second fundamental form in a real space form of non-positive curvature must be totally geodesic (\cite{t}).
\end{proof}

\begin{rmk}\rm Clearly, in the case of $p=1$, ${\rm LS}(m_1,p_1,r,\mu)$ will reduce to those hypersurfaces first introduced in \cite{lz06} (See Examples 3.1 and 3.2 there).\end{rmk}

\section{Proof of the main theorem}

Let $x:M^m\to \bbs^{m+p}$ be an umbilic-free submanifold in $\bbs^{m+p}$ satisfying all the conditions in the main theorem, and $\lambda_1,\lambda_2$ be the two different Blaschke eigenvalues of $x$. By the assumption that the \mo form vanishes identically and the Blaschke tensor $A$ is parallel, it is not hard to see that $(M,g)$ can be decomposed into a direct product of two Riemannian manifolds $(M_1,g^{(1)})$ and $(M_2,g^{(2)})$ with $m_1:=\dim M_1$ and $m_2:=\dim M_2$, that is,
$$(M^m,g)=(M_1,g^{(1)})\times(M_2,g^{(2)}),$$
such that, by choosing the orthonormal frame field $\{E_i\}$ of $(M^m,g)$ satisfying
$$E_1,\cdots,E_{m_1}\in TM_1,\ E_{m_1+1},\cdots,E_m\in TM_2,$$
the components $A_{ij}$ of $A$ with respect to $\{E_i\}$ are diagonalized as follows:
\be\label{4-1}
A_{i_1j_1}=\lambda_1\delta_{i_1j_1},\ A_{i_2j_2}=\lambda_2\delta_{i_2j_2},\ A_{i_1j_2}=A_{i_2j_1}=0,
\ee
where and from now on we agree with
$$
1\leq i_1,j_1,k_1,\cdots\leq m_1,\quad m_1+1\leq i_2,j_2,k_2,\cdots\leq m.
$$
Since the \mo form $C\equiv 0$, we can also assume by \eqref{2.16} that the corresponding components $B^\alpha_{ij}$ of the \mo second fundamental form $B$ satisfy
\be\label{4-2}
B^\alpha_{i_1j_2}\equiv 0,\mb{ for all }\alpha,i_1,j_2.
\ee

In general, we have
\begin{lem}\label{lem4.1} It holds that
\be\label{4-2.1}
B^\alpha_{ij\cdots k}\equiv 0,\mb{ if neither }1\leq i,j,\cdots,k\leq m_1\mb{ nor }m_1+1\leq i,j,\cdots,k\leq m.
\ee
where $ij\cdots k$ is a multiple index of order no less than $2$.
\end{lem}

\begin{proof}
Due to \eqref{4-2} and the method of induction, it suffices to prove that if \eqref{4-2.1} holds then
\be\label{4-2.2}
B^\alpha_{ij\cdots kl}\equiv 0,\mb{ if neither }1\leq i,j,\cdots,k,l\leq m_1\mb{ nor }m_1+1\leq i,j,\cdots,k,l\leq m.
\ee
In fact, we only need to consider the following two cases:

(i) Neither $1\leq i,j,\cdots,k\leq m_1$ nor $m_1+1\leq i,j,\cdots,k\leq m$.

In this case we use \eqref{4-2.1} and $\omega^{j_2}_{i_1}=0$ to find
$$
B^\alpha_{ij\cdots kl}\omega^l=dB^\alpha_{ij\cdots k}-\sum B^\alpha_{lj\cdots k}\omega^l_i
-\sum B^\alpha_{il\cdots k}\omega^l_j-\cdots -\sum B^\alpha_{ij\cdots l}\omega^l_k
+\sum B^\beta_{ij\cdots k}\omega^\alpha_\beta\\
\equiv 0.
$$
So \eqref{4-2.2} is true.

(ii) $Either 1\leq i,j,\cdots,k\leq m_1$ or $m_1+1\leq i,j,\cdots,k\leq m$.

Without loss of generality, we assume the first. Then it must be that $m_1+1\leq l\leq m$.
Note that by \eqref{2-16} and \eqref{4-2},
\be\label{4-2.3}
R^\bot_{\alpha\beta i_1j_2}=\sum(B^\alpha_{j_2q}B^\beta_{i_1q} -B^\alpha_{i_1q}B^\beta_{j_2q})\equiv 0,\quad\forall i_1,j_2. \ee
This together with Case (i), the Ricci identities \eqref{2-21} and the fact that $R_{i_1j_2ij}\equiv 0$ shows that
$$
B^\alpha_{ij\cdots kl}=B^\alpha_{ij\cdots lk}+\sum B^\alpha_{qj\cdots}R_{iqkl} +\sum B^\alpha_{iq\cdots}R_{jqkl}+\cdots-\sum B^\beta_{ij\cdots}R^\bot_{\beta\alpha kl}.
\equiv 0.
$$
\end{proof}

\begin{lem}\label{lem4-1}It holds that, for all $i_1,j_1,k_1,\cdots,l_1$ and $i_2,j_2,\cdots,k_2$,
\begin{align}
&\sum_\alpha B^\alpha_{i_1j_1}B^\alpha_{i_2j_2} =-(\lambda_1+\lambda_2)\delta_{i_1j_1}\delta_{i_2j_2},\label{4-3}\\
&\sum_\alpha B^\alpha_{i_1j_1k_1}B^\alpha_{i_2j_2} =0,\quad \sum_\alpha B^\alpha_{i_1j_1}B^\alpha_{i_2j_2k_2} =0.\label{4-4}
\end{align}
More generally,
\be\label{4-5}
B^\alpha_{i_1j_1k_1\cdots l_1}B^\alpha_{i_2j_2\cdots k_2} =0,\quad
B^\alpha_{i_1j_1\cdots k_1}B^\alpha_{i_2j_2k_2\cdots l_2} =0,
\ee
where $i_1j_1k_1\cdots l_1$ and $i_2j_2k_2\cdots l_2$ are multiple indices of order no less than $3$.
\end{lem}

\begin{proof} This lemma comes mainly from the \mo Gauss equation \eqref{2.15} and the parallel assumption of the Blaschke tensor $A$. In fact, \eqref{4-3} is given by \eqref{2.15}, \eqref{4-1}, \eqref{4-2} and that $R_{i_1i_2j_2 j_1}\equiv 0$; \eqref{4-4} is given by \eqref{2.15}, \eqref{4-2.1}, $R_{i_1i_2j_2 j_1}\equiv 0$ and the parallel of $A$; Finally, \eqref{4-5} can be shown by the method of induction using Lemma \ref{lem4.1}.
\end{proof}

As the corollary of \eqref{4-3}, we have
\begin{align}
&\sum_\alpha B^\alpha_{i_1j_1}(B^\alpha_{i_2i_2}-B^\alpha_{j_2j_2}) =\sum_\alpha (B^\alpha_{i_1i_1}-B^\alpha_{j_1j_1})B^\alpha_{i_2j_2}=0,\label{4-6}\\
&\sum_\alpha B^\alpha_{i_1j_1}B^\alpha_{i_2j_2}=0,\mb{ if }i_1\neq j_1 \mb{ or }i_2\neq j_2.\label{4-7}
\end{align}

Define
\begin{align}
&V_1=\spn\left\{\sum_\alpha B^\alpha_{i_1j_1\cdots k_1}E_\alpha\right\},\quad V_2=\spn\left\{\sum_\alpha B^\alpha_{i_2j_2\cdots k_2}E_\alpha\right\};\\
&V_{10}=V_1\cap(V_2)^\bot,\quad V_{20}=V_2\cap(V_1)^\bot,\quad\mb{ so that }\quad V_{10}\,\bot\, V_{20}.
\label{4-8}
\end{align}

Let $V'_0$ (resp. $V''_0$) be the orthogonal complement of $V_{10}$ in $V_1$ (resp. $V_{20}$ in $V_2$).

\begin{lem}\label{lem4.2} It holds that $V'_0=V''_0$.\end{lem}

\begin{proof}
For any $i,j$, we denote by $B^{V'_0}_{ij}$ (resp. $B^{V''_0}_{ij}$) the $V'_0$-component (resp. the $V''_0$-component) of $B_{ij}$. Then, for any $i_1,j_1,i_2,j_2$, it follows from \eqref{4-6} and \eqref{4-7} that
\be\label{4.12}
B^{V'_0}_{i_1i_2}=B^{V''_0}_{j_1j_2}=0,\quad
B^{V'_0}_{i_1i_1}=B^{V'_0}_{j_1j_1},\quad
B^{V''_0}_{i_2i_2}=B^{V''_0}_{j_2j_2}.
\ee
So that
\be\label{4.13}
V'_0=\spn\{B^{V'_0}_{i_1i_1}\},\quad V''_0=\spn\{B^{V''_0}_{j_2j_2}\}.
\ee

On the other hand, by the second equation in \eqref{2.14}, we have
\be\label{4.14}
\sum_{i_1}B_{i_1i_1}+\sum_{j_2}B_{j_2j_2}=0.
\ee
But, for any $i_1,j_2$,
$$
B_{i_1i_1}=B^{V_{10}}_{i_1i_1}+B^{V'_0}_{i_1i_1},\quad
B_{j_2j_2}=B^{V_{20}}_{j_2j_2}+B^{V''_0}_{j_2j_2}.
$$
Since $V_{10}\,\bot V_{20}$, \eqref{4.14} reduces to
\be\label{4.15}
\sum_{i_1}B^{V_{10}}_{i_1i_1} =\sum_{j_2}B^{V_{20}}_{j_2j_2}=0,\quad
\sum_{i_1}B^{V'_0}_{i_1i_1}+\sum_{j_2}B^{V''_0}_{j_2j_2}=0.
\ee
The last equality in \eqref{4.15} together with \eqref{4.12} shows that, for some $i_1,j_2$,
$$m_1B^{V'_0}_{i_1i_1}+m_2B^{V''_0}_{j_2j_2}=0$$
which with \eqref{4.13} proves that $V'_0=V''_0:=V_0$.
\end{proof}

Remark that, by \eqref{4.13}, we have $\dim V_0\leq 1$. Now we need to consider the following two cases:

Case 1. $\lambda_1+\lambda_2\neq 0$.

In this case, by \eqref{4-6} and \eqref{4-7}, it is not hard to see that $\dim V_0=1$ and thus, locally, we can choose an orthonormal normal frame field $\{E_\alpha\}$ such that $V_0=\bbr E_{\alpha_0}$ for some $\alpha_0\in \{m+1,\cdots,m+p\}$. Furthermore, it holds that $\lambda_1+\lambda_2>0$ and
\begin{align}
B^{\alpha_0}_1:=&B^{\alpha_0}_{i_1i_1}=\cdots=B^{\alpha_0}_{j_1j_1} =\pm\sqrt{\fr{m_2}{m_1}(\lambda_1+\lambda_2)},\quad B^{\alpha_0}_{i_1j_1}=0\mb{ for }i_1\neq j_1; \label{4-9}\\ B^{\alpha_0}_2:=&B^{\alpha_0}_{i_2i_2}=\cdots=B^{\alpha_0}_{j_2j_2} =\mp\sqrt{\fr{m_1}{m_2}(\lambda_1+\lambda_2)},\quad B^{\alpha_0}_{i_2j_2}=0\mb{ for }i_2\neq j_2.\label{4-10}
\end{align}
By altering the direction of $E_{\alpha_0}$ if necessary, we can assume without loss of generality that $B^{\alpha_0}_1\geq 0$. Hence we have
\be\label{4-11}
B^{\alpha_0}_1=\sqrt{\fr{m_2}{m_1}(\lambda_1+\lambda_2)}>0,\quad B^{\alpha_0}_2=-\sqrt{\fr{m_1}{m_2}(\lambda_1+\lambda_2)}<0.
\ee

In what follows, we shall agree with the following notation:
\be\label{4-12}
E_{\alpha_1},E_{\beta_1},E_{\gamma_1},\cdots \in V_{10},\quad E_{\alpha_2},E_{\beta_2},E_{\gamma_2},\cdots  \in V_{20}.
\ee
Since, for any $i_1,j_1$ (resp. $i_2,j_2$), the orthogonal projection $B^{V_{20}}_{i_1j_1}$ of $B_{i_1j_1}$ to $V_{20}$ (resp. $B^{V_{10}}_{i_2j_2}$ of $B_{i_2j_2}$ to $V_{10}$) vanishes identically, we have
\be\label{4-13}
B^{\alpha_2}_{i_1j_1}=B^{\alpha_1}_{i_2j_2}=0,\quad \forall i_1,j_1,i_2,j_2.
\ee

The following lemma can be shown by Lemma \ref{lem4-1} and \eqref{4-7} using the method of induction:

\begin{lem}\label{lem4-2} There exist suitably chosen frames $\{E_{\alpha_1}\}$ and $\{E_{\alpha_2}\}$ for $V_{10}$ and $V_{20}$, respectively, such that the \mo normal connection forms $\omega^\beta_\alpha$ with respect to the frame $\{E_\alpha\}$ satisfy
\be\label{4-14}
\omega^{\beta_2}_{\alpha_1}=0,\quad \omega^\beta_{\alpha_0}=0.
\ee
\end{lem}

Let $Y$ and $N$ be the \mo position vector and the \mo biposition vector of $x$, respectively. Motivated by \cite{lwz}, see also \cite{lw1}, \cite{lz05} and the very recent paper \cite{zhw}, we define another vector-valued function
\be\label{4-15}
{\mathbf c}:=N+aY+bE_{\alpha_0}
\ee
for some constants $a$, $b$ to be determined. Then by using \eqref{2-6} and \eqref{2-8} we find
$$
d{\mathbf c}=(a+\lambda_1-bB^{\alpha_0}_1)\omega^{i_1}Y_{i_1}
+(a+\lambda_2-bB^{\alpha_0}_2)\omega^{i_2}Y_{i_2}.
$$
Since
$$B^{\alpha_0}_1-B^{\alpha_0}_2 =\sqrt{\fr{m_2}{m_1}(\lambda_1+\lambda_2)}+\sqrt{\fr{m_1}{m_2}(\lambda_1+\lambda_2)} =m\sqrt{\fr{\lambda_1+\lambda_2}{m_1m_2}}>0,
$$
the system of linear equations
$$
a+\lambda_1-bB^{\alpha_0}_1=0,\quad a+\lambda_2-bB^{\alpha_0}_2=0
$$
for $a$, $b$ has a unique solution as
\be\label{4-16}
a=-\fr{m_1\lambda_1+m_2\lambda_2}{m},\quad b=\fr{\lambda_1-\lambda_2}{m}\sqrt{\fr{m_1m_2}{\lambda_1+\lambda_2}}.
\ee
Thus the following lemma is proved:

\begin{lem}\label{lem4-3} Let $a$, $b$ be given by \eqref{4-16}. Then the vector-valued function ${\mathbf c}$ defined by \eqref{4-15} is constant on $M^m$ and
\be\label{4-17}
\lagl {\mathbf c},{\mathbf c}\ragl_1=2a+b^2,\quad \lagl {\mathbf c},Y\ragl_1=1.
\ee
\end{lem}

Next we consider separately the following three subcases:

Subcase (1): ${\mathbf c}$ is time-like, that is, $\lagl {\mathbf
c},{\mathbf c}\ragl_1=-r^2<0$ for some positive number $r$. In
this case, there exists a $T\in O^+(m+p+1,1)$ such that $\td{\mathbf
c}:=T({\mathbf c})=(-r,0)$ with $0\in\bbr^{m+p+1}$. So
\be\label{4.23}
\td {\mathbf c}=(-r,0)=T(N)+aT(Y)+bT(E_{\alpha_0}).
\ee
If we write $\td Y:=T(Y)=(\td Y_0,\td Y_1)$ with $\td Y_1\in\bbr^{m+p+1}$, and let $\td x:M^m\to \bbs^{m+p}$ be the immersion with $\td Y$ as its \mo position vector, then $\td Y_0>0$ and $x$ is \mo equivalent to $\td x$ for which the \mo factor $\td \rho=\td Y_0$. From \eqref{4-17} and \eqref{4.23} we find
$$r\td \rho=r\td Y_0=\lagl \td{\mathbf c},\td Y\ragl_1=1.$$
Hence
\be\label{4-18}\td \rho=r^{-1}=\const.\ee

On the other hand, since the \mo form $\td C$ of $\td x$ vanishes
identically, we know from \eqref{2.11} and \eqref{4-18} that $\td
H^\alpha_{,i}=0$, that is, the mean curvature vector of $\td x$ is
parallel. In particular, the mean curvature $|\td H|$ of $\td x$ is also constant. This with \eqref{2.12}, \eqref{4-18} shows that the second fundamental form $\td H\cdot \td h=\td H^\alpha\td h^\alpha$ of $\td x$ in the direction of mean curvature $\td H$ is parallel, that is, $\td x$ is pseudo-parallel.

On the other hand, the \mo metric is the same as that of $x$ which can be written as $g=r^{-2}d\td x\cdot d\td x$. Therefore, if $\td R$ is the scalar curvature of the metric
$d\td x\cdot d\td x$, then
$$\td R=m(m-1)r^{-2}\kappa=\const$$ since, by \eqref{2.14} and
$$\tr A=m_1\lambda_1+m_2\lambda_2=\const,$$
$\kappa$ is constant.

Furthermore, since the Blaschke tensor $A$ of $x$, which is equivalent to that of $\td x$, has two distinct eigenvalues,  it follows from \eqref{2.12} that $\td H\cdot\td h$ must has two distinct eigenvalues, or the same, $\td x$ has two distinct principal curvatures in the direction of the mean curvature vector.

Subcase (2): ${\mathbf c}$ is light-like, that is, $\lagl {\mathbf
c},{\mathbf c}\ragl_1=0$. In this case, there exists a $T\in
O^+(m+p+1,1)$ such that $\td{\mathbf c}:=T({\mathbf c})=(-1,1,0)$
with $0\in\bbr^{m+p}$. So
\be\label{4.26}
\td {\mathbf c}=(-1,1,0)=T(N)+aT(Y)+bT(E_{\alpha_0}).
\ee
If we write $\td Y=T(Y)=(\td Y_0,\td Y_1)$, then $\td Y_0>0$.
Define $\td x=\td Y^{-1}_0\td Y_1$, then
\be\label{4.27}
\td Y=\td\rho(1,\td x),\mb{\ where\ }\td\rho=\td Y_0.
\ee
It is clear that $\td x: M^m\to\bbs^{m+p}$ is an immersion with
$\td Y$ as its \mo position vector. Therefore $x$ is \mo
equivalent to $\td x$. Write $\td
x=(\td x_0,\td x_1)$ in which $\td x_1\in\bbr^{m+p}$. Then by \eqref{4.27},
$\td Y=(\td\rho,\td\rho\td x_0,\td\rho\td x_1)$. From \eqref{4-17} and
\eqref{4.26} we know that $\lagl\td Y,\td{\mathbf c}\ragl_1=1$, i.e.
$\td\rho(1+\td x_0)=1$. So $1+\td x_0> 0$ and $\td\rho=(1+\td
x_0)^{-1}$. In particular, $\td x_0> -1$. This indicates that
$\td x(M^m)\subset \bbs^{m+p}\bsl\{(-1,0)\}$, so that we can consider the pre-image under $\sigma$ of $\td x$, that is, $\bar x:=\sigma^{-1}\circ\td
x$. Then $\bar x:M^m\to \bbr^{m+p}$ is an umbilic-free immersion since $\td x$ is.

By the definition \eqref{1.1} of $\sigma$, one sees that
$$
\td x_0=\fr{1-|\bar x|^2}{1+|\bar x|^2},\quad\td x_1=\fr{2\bar x}{1+|\bar x|^2}.
$$
Thus
\be\label{4.28}\td\rho=(1+\td x_0)^{-1}=\fr12(1+|\bar x|^2),\ee
so that
\be\label{4.29}\td Y=\td \rho(1,\td x)=\fr12(1+|\bar x|^2,1-|\bar x|^2,2\bar x).\ee
From
\eqref{4.29} one easily finds that the \mo metric $g$ can be written as
$$g=\lagl d\td Y,d\td Y\ragl_1=d\bar x\cdot d\bar x.$$
Using a theorem of Liu-Wang-Zhao in \cite{lwz} we see that the \mo factor $\ol\rho$ of $\bar x$
(cf. \eqref{2.1}) is identical to the constant $1$. Clearly, the scalar curvature
$\ol R$ of $\bar x$ is a constant since, once more, $\kappa$ is constant. Again, by \cite{lwz} we know that all the components $\ol
C^\alpha_i$ of the \mo form $\ol C$ of $\bar x$ vanish. It then follows by \eqref{2.11} and Remark \ref{rmk2} that the mean curvature vector field $\ol H$ of $\bar x$ is parallel. Finally, as in the subcase (1), it follows easily from \eqref{2.12'} that $\bar x$ is pseudo-parallel and has two distinct principal curvatures in the direction of the mean curvature vector.

Subcase (3): ${\mathbf c}$ is space-like, that is, $\lagl {\mathbf
c},{\mathbf c}\ragl_1=r^2>0$ for some $r>0$. In this case, there
exists a $T\in O^+(m+p+1,1)$ such that
$$\td{\mathbf
c}:=T({\mathbf c})=(0,r,0)\in\bbr_1\times\bbr\times\bbr^{m+p}\equiv
\bbr^{m+p+2}_1.$$
So
\be\label{4.30}
\td {\mathbf c}=(0,r,0)=T(N)+aT(Y)+bT(E_{\alpha_0}).
\ee
Similar to the subcase (2), if we write $\td Y=T(Y)=(\td Y_0,\td Y_1)$ with $\td Y_1\in\bbr^{m+p+1}$, then $\td Y_0>0$ and we can define $\td x=\td Y^{-1}_0\td Y_1$, so that
\be\label{4.31}
\td Y=\td\rho(1,\td x),\mb{\ where\ }\td\rho=\td Y_0.
\ee
Once again, $\td x:M^m\to\bbs^{m+p}$ is an immersion with
$\td Y$ as its \mo position vector. Therefore $x$ is \mo
equivalent to the immersion $\td x$. Write $\td
x=(\td x_0,\td x_1)$ for some $\td x_1\in\bbr^{m+p}$. Then by \eqref{4.31},
$\td Y=(\td\rho,\td\rho\td x_0,\td\rho\td x_1)$. From \eqref{4-17} and
\eqref{4.30} we know that $\lagl\td Y,\td{\mathbf c}\ragl_1=1$ and thus
$r\td\rho\td x_0=1$, implying $\td x_0> 0$. Hence $\td
x(M^m)\subset \bbs^{m+p}_+$. By defining $\bar x:=\tau^{-1}\circ\td x$, the pre-image of $\td x$ under $\tau$, we obtain an immersion $\bar x:M^m\to \bbh^{m+p}$ which is umbilic-free since $\td x$ is.

Write $\bar x=(\bar x_0,\bar x_1)$ with $\bar x_1\in \bbr^{m+p}$. Then by the definition \eqref{1.2}
of $\tau$, one sees that
\be\label{4.32}\td x_0=\bar x^{-1}_0,\quad\td x_1=\bar x^{-1}_0\bar x_1.\ee
Thus $\td\rho=(r\td x_0)^{-1}=r^{-1}\bar x_0$ so that
\be\label{4.33}\td Y=\td \rho(1,\td x)=r^{-1}(\bar x_0,1,\bar x_1).\ee
From \eqref{4.33} one easily finds
that the \mo metric $g$ can be written as
\be\label{4.34} g=\lagl d\td Y,d\td Y\ragl_1=r^{-2}\lagl d\bar x,d\bar x\ragl_1.\ee
By \cite{lwz} we see that the \mo factor $\ol\rho$ of $\bar x$ (cf.
\eqref{2.1}) is identical to $r^{-1}$. Then, similar to the subcase (2), we can use
\cite{lwz} and \eqref{2.12''} to show that $\bar x$ is pseudo-parallel with constant scalar curvature, parallel mean curvature vector field, and has two distinct principal curvatures in the direction of the mean curvature vector.

Case 2. $\lambda_1+\lambda_2=0$.

In this case, we can assume that $\lambda_1<0$ and thus $\lambda_2>0$.
Define a positive number $r$ by $r^{-2}=2\lambda_2$. Then $2\lambda_1=-r^{-2}$.

By Lemma \ref{lem4-1} it is readily that $\dim V_0=0$, so that  $V_1=V_{10}$, $V_2=V_{20}$. In particular, $V_1 \bot V_2$ and thus the \mo normal bundle $V\to M^m$ splits orthogonally into direct sum of $V_1$ and $V_2$:
$V=V_1\oplus V_2$. Furthermore, by Lemma \ref{lem4-1}, there is an orthonormal normal frame field $\{E_\alpha\}$ such that the \mo normal connection forms $\omega^\beta_\alpha$ meet
\be\label{4-35}
\omega^{\beta_2}_{\alpha_1}\equiv 0, \mb{ where }E_{\alpha_1},E_{\beta_1},\cdots\in V_1,\quad E_{\alpha_2},E_{\beta_2},\cdots\in V_2.
\ee

Define
$$
y=-\fr1{2\lambda_2}(N-\lambda_2 Y),\quad
y_2=\fr1{2\lambda_2}(N+\lambda_2 Y).
$$
Then $y+y_2=Y$ and
\be\label{4.38}
\lagl y,y\ragl_1=-\fr1{2\lambda_2}=-r^2<0,\quad \lagl y_2,y_2\ragl_1=\fr1{2\lambda_2}=r^2>0.
\ee
Furthermore, by \eqref{2-6}, \eqref{4-1} and $\lambda_1+\lambda_2=0$, we find that
\begin{align}
&dy=-\fr1{2\lambda_2}\left(\sum_{i,j}A_{ij}\omega^jY_i-\lambda_2\sum_i\omega^iY_i\right) =\sum_{i_1}\omega^{i_1}Y_{i_1};\label{4.39}\\
&dy_2=\fr1{2\lambda_2}\left(\sum_{i,j}A_{ij}\omega^jY_i-\lambda_1\sum_i\omega^iY_i\right) =\sum_{i_2}\omega^{i_2}Y_{i_2}.\label{4.40}
\end{align}
Thus $y$ and $y_2$ is constant on $M_2$ and $M_1$, respectively.

Using \eqref{2-7}, \eqref{2-8}, \eqref{4-35}, \eqref{4.39} and \eqref{4.40}, we can easily obtain the following conclusion:

\begin{cor}\label{lem4-4} The subbundles $V_1$ and $V_2$ are parallel in the \mo normal bundle $V$, and the \mo normal connection on $V$ is the direct sum of its restriction on $V_1$ and its restriction on $V_2$. Moreover,
$$\bbr y\oplus V_1\oplus TM_1,\quad\bbr y_2\oplus V_2\oplus TM_2$$
are orthogonal to each other in $\bbr^{m+p+2}_1$ and constant on $M_1$ and $M_2$, respectively.
\end{cor}

\begin{rmk}\rm By Corollary \ref{lem4-4} and \eqref{4.38}, there exists an element $T\in O^+(m+p+1,1)$, such that
\be
T(\bbr y\oplus V_1\oplus TM_1)=\bbr^{m_1+p_1+1}_1,\quad
T(\bbr y_2\oplus V_2\oplus TM_2)=\bbr^{m_2+p_2+1}.
\ee
\end{rmk}

 Corollary \ref{lem4-4} implies that, by restriction, we can identify the subbundle $V_1$ (resp. $V_2$) of $V$ with a vector bundles $V_1\to M_1$ on $M_1$ (resp. $V_2\to M_2$ on $M_2$). In this sense, the \mo normal bundle $V\to M^m$ with the \mo normal connection is the direct product of $V_1\to M_1$ and $V_2\to M_2$ with their induced connections.

Now from the \mo second fundamental form $B$, we define
$$
\Ba=\sum B^{\alpha_1}_{i_1j_1}\omega^{i_1}\omega^{j_1}E_{\alpha_1},\quad
\Bb=\sum B^{\alpha_2}_{i_2j_2}\omega^{i_2}\omega^{j_2}E_{\alpha_2}.
$$
Then $\Ba$ (resp. $\Bb$) is a $V_1$-valued (resp.  $V_2$-valued) symmetric $2$-form on $M_1$ (resp. on $M_2$) with components $\Ba^{\alpha_1}_{i_1j_1}=B^{\alpha_1}_{i_1j_1}$ (resp. $\Bb^{\alpha_2}_{i_2j_2}=B^{\alpha_2}_{i_2j_2}$).

Let $\Ba^{\alpha_1}_{i_1j_1,k_1}$ (resp. $\Bb^{\alpha_2}_{i_2j_2,k_2}$) be the components of the covariant derivative of $\Ba$ (resp. $\Bb$). Then, as the consequence of \eqref{4-2}, \eqref{4-35} and
Corollary \ref{lem4-4}, we have
\be\label{4-36}
\Ba^{\alpha_1}_{i_1j_1,k_1}=B^{\alpha_1}_{i_1j_1k_1},\quad \Bb^{\alpha_2}_{i_2j_2,k_2}=B^{\alpha_1}_{i_2j_2k_2}.
\ee

Since $B^{\alpha_2}_{i_1j_1}=B^{\alpha_1}_{i_2j_2}=0$, the vanishing of the \mo form $C$ together with \eqref{2.15}, \eqref{4-1}, \eqref{2.18}, \eqref{4-2}, \eqref{4-36} and \eqref{2-16} proves t easily he following lemma:

\begin{lem}\label{lem4-5}
The Riemannian manifold $(M_1,g^{(1)})$ (resp. $(M_2,g^{(2)})$) and the vector bundle valued symmetric tensor $\Ba$ (resp. $\Bb$) satisfies the Gauss equation, Codazzi equation and Ricci equation for submanifolds in a space form of constant curvature $2\lambda_1$ (resp. $2\lambda_2$). Namely
\begin{align}
R_{i_1j_1k_1l_1}=&\sum (\Ba^{\alpha_1}_{i_1l_1}\Ba^{\alpha_1}_{j_1k_1} -\Ba^{\alpha_1}_{i_1k_1}\Ba^{\alpha_1}_{j_1l_1}) +2\lambda_1(\delta_{i_1l_1}\delta_{j_1k_1}-\delta_{i_1k_1}\delta_{j_1l_1}),\label{4-37}\\
R_{i_2j_2k_2l_2}=&\sum (\Bb^{\alpha_2}_{i_2l_2}\Bb^{\alpha_2}_{j_2k_2} -\Bb^{\alpha_2}_{i_2k_2}\Bb^{\alpha_2}_{j_2l_2}) +2\lambda_2(\delta_{i_2l_2}\delta_{j_2k_2}-\delta_{i_2k_2}\delta_{j_2l_2}),\label{4-38}\\
\Ba^{\alpha_1}_{i_1j_1,k_1}=&\Ba^{\alpha_1}_{i_1k_1,j_1},\quad \Bb^{\alpha_2}_{i_2j_2,k_2}=\Bb^{\alpha_2}_{i_2k_2,j_2},\\
R^\bot_{\alpha_1\beta_1i_1j_1}=&\sum (\Ba^{\alpha_1}_{j_1k_1}\Ba^{\beta_1}_{i_1k_1} -\Ba^{\alpha_1}_{i_1k_1}\Ba^{\beta_1}_{j_1k_1}),\\
R^\bot_{\alpha_2\beta_2i_2j_2}=&\sum (\Bb^{\alpha_2}_{j_2k_2}\Bb^{\beta_2}_{i_2k_2} -\Bb^{\alpha_2}_{i_2k_2}\Bb^{\beta_2}_{j_2k_2}).
\end{align}
\end{lem}

Since $2\lambda_1=-\fr1{r^2}<0$, $2\lambda_2=\fr1{r^2}>0$, Lemma \ref{lem4-5} shows that there exist an isometric immersion
$$\td y\equiv(\td y_0,\td y_1):(M_1,g^{(1)})\to\bbh^{m_1+p_1}\left(-\fr1{r^2}\right)\subset\bbr^{m_1+p_1+1}_1$$
with $\Ba$ as its second fundamental form, and an isometric immersion
$$\td y_2:(M_2,g^{(2)})\to\bbs^{m_2+p_2}(r)\subset\bbr^{m_2+p_2+1}$$
with $\Bb$ as its second fundamental form.

Note that $B^{\alpha_2}_{i_1j_1}=B^{\alpha_1}_{i_2j_2}\equiv 0$. It follows from \eqref{2.14} that both $\td y$ and $\td y_2$ are minimal immersions. Furthermore, if $\td S_1$ and $\td S_2$ denote, respectively, the scalar curvatures of $M_1$ and $M_2$, then by \eqref{4-37}, \eqref{4-38} and the minimality, we have
\be\label{4-39}
\td S_1=-\fr{m_1(m_1-1)}{r^2}-\sum (B^{\alpha_1}_{i_1j_1})^2,\quad \td S_2=\fr{m_2(m_2-1)}{r^2}-\sum (B^{\alpha_2}_{i_2j_2})^2,
\ee
showing that
\begin{align}
\td S_1+\fr{m_1(m_1-1)}{r^2}=-\sum(B^{\alpha_1}_{i_1j_1})^2 \leq 0,\label{4-39.1}\\ \td S_2-\fr{m_2(m_2-1)}{r^2}=-\sum(B^{\alpha_2}_{i_2j_2})^2 \leq 0.\label{4-39.2}
\end{align}
Thus by \eqref{2.14},
\begin{align}
\td S_1+\td S_2=&\fr{-m_1(m_1-1)+m_2(m_2-1)}{r^2}-\fr{m-1}{m}=\const. \label{4-40}
\end{align}

Since $\td S_1$ and $\td S_2$ are functions defined on $M_1$ and $M_2$, respectively, it follows that both $\td S_1$ and $\td S_2$ are constant and, by \eqref{4-39.1}, \eqref{4-39.2}
\be\label{s1s2}
\td S_1=-\fr{m_1(m_1-1)}{r^2}-\fr{m-1}{m}\mu,\quad \td S_2=\fr{m_2(m_2-1)}{r^2}-\fr{m-1}{m}(1-\mu)
\ee
for some constant $\mu\in [0,1]$.

Now we are in a position to complete the proof of the main theorem (Theorem \ref{main}).

As discussed earlier in this section, there are only the following two cases that need to be considered:

If $\lambda_1+\lambda_2\neq 0$, then $x$ is \mo equivalent to the three kinds of submanifolds (1), (2) and (3) as listed in Theorem \ref{main}, according to the Subcases (1), (2) and (3), respectively;

If $\lambda_1+\lambda_2=0$, then have two immersions
$$\td y:(M_1,g^{(1)})\to \bbh^{m_1+p_1}\left(-\fr1{r^2}\right),\quad \td y_2:(M_2,g^{(2)})\to \bbs^{m_2+p_2}(r),$$
which are minimal and, by \eqref{s1s2}, have constant scalar curvatures $\td S_1$, $\td S_1$, respectively.

Let ${\rm LS}(m_1,p_1,r,\mu)$ be one of the submanifolds in Example \ref{expl3.2} defined by $\td y$ and $\td y_2$. Then it is not hard to see that ${\rm LS}(m_1,p_1,r,\mu)$ has the same \mo metric $g$ and the same \mo second fundamental form $B$ as those of $x$. Furthermore, by choosing the normal frame field $\{\td e_\alpha\}$ as given in \eqref{3-10} and \eqref{3-11} where, in the present case,
$$\bar e_\alpha=E_\alpha,\quad m+1\leq \alpha\leq m+p,$$
we compute directly (cf. \eqref{3-37}):
$$
\td\omega^\beta_\alpha=d\td e_\alpha\cdot\td e_\beta=\lagl dE_\alpha,E_\beta\ragl_1=\begin{cases}\omega^\beta_\alpha,& \mb{for both } m+1\leq\alpha,\beta\leq m+p_1,\\
&\mb{and } m+p_1+1\leq\alpha,\beta\leq m+p;\\
0,& otherwise,
\end{cases}
$$
implying that $x$ and ${\rm LS}(m_1,p_1,r,\mu)$ have the same \mo normal connection. Therefore, by Theorem \ref{wth}, $x$ is \mo equivalent to ${\rm LS}(m_1,p_1,r,\mu)$ and we are done.

\end{document}